\documentclass[10pt]{article}
\newtheorem{theorem}{Theorem}

\newtheorem{lemma}{Lemma}
\newtheorem{example}{Example}
\newtheorem{corollary}{Corollary}

\newtheorem{question}{Question}
\newtheorem{conjecture}{Conjecture}
\newtheorem{prelem}{{\bf Theorem}}

\newenvironment{oldtheorem}{\begin{prelem}{\hspace{-0.5
em}{\bf}}}{\end{prelem}}
\input amssym.tex
\newtheorem{preproof}{{\bf Proof.}}

\newenvironment{proof}[1]{\begin{preproof}{\rm #1}\hfill{\rule[-0.5mm]
                         {2mm}{2mm}}}{\end{preproof}}

\def\n#1{\vbox to 3mm{\vspace{1mm}\vfill \hbox to 2.0mm{\hfill
             $#1$\hfill} \vfill }}
\def\m#1#2{\raise 0.2ex\hbox{
    ${#1_{\bf \displaystyle #2}}$}}
\def\x#1{\raise 0.5ex\hbox{
    ${#1}$}}
\def\arraystretch{1.0}                 

\usepackage{array}
\title{{\bf On the existence of $k$-homogeneous \\ Latin bitrades}}
\author{Behrooz Bagheri Gh. and
E. S. Mahmoodian}
\date{}

\begin{document}
\maketitle
\begin{center}
Department of Mathematical Sciences \\
Sharif University of Technology \\
P. O. Box 11155--9415,
Tehran, I. R. Iran
\end{center}
%
\begin{abstract}
Let $T$ be a partial Latin square and $L$ a Latin square such
that $T\subseteq L$. Then $T$ is called a {\sf Latin trade}, if
there exists a partial Latin square  $T^{*}$ such that $T^{*}\cap
T=\emptyset$ and $(L\backslash T)\cup T^{*}$ is a Latin square. We call
$T^{*}$ a {\sf disjoint mate} of $T$ and the pair $(T,T^{*})$ is called a
{\sf Latin bitrade}.
A Latin bitrade where empty rows and  columns are ignored,
 is called a
{\sf $k$--homogeneous Latin bitrade}, if in each row and each
column it contains exactly $k$ elements, and each element appears
exactly $k$ times. The number of filled cells in a Latin trade is
referred to as its {\sf volume}.

Following the earlier work on $k$--homogeneous Latin bitrades by
 Cavenagh, Donovan, and  Dr\'{a}pal (2003 and 2004) Bean,
 Bidkhori,  Khosravi, and
E. S. Mahmoodian (2005) we prove the following results.

All $k$--homogeneous Latin bitrades of
volume $km$ exist, for
\begin{itemize}
\item
all odd integers $k$ and $m \ge k$,
\item
all even integers $k>2$ and $m \ge \min\{k+u, \frac{3k}{2}\}$,
where $u$ is any odd integer which divides $k$,
\item
all $m\geq k$, where $3\le k\leq 37$.
\end{itemize}
\end{abstract}

\section{Introduction}        

Two disjoint partial Latin squares $T$ and $T^{*}$ of the same order, with the same set of filled cells and satisfying the property that
corresponding rows (corresponding columns) contain the same entry values, form a {\sf Latin trade} and its disjoint mate. The pair $(T,T^{*})$ is called a
{\sf Latin bitrade}. In earlier papers the word ``Latin trade'' is used for ``Latin bitrade'', but we keep the word ``trade'' for each partial Latin square of a Latin bitrade.
The
study of Latin trades and combinatorial trades in general, has generated much interest in recent years. For a
survey on the topic see \cite{MR2041871},
 \cite{MR2048415},  and \cite{CavenaghMathSlovac}.

A Latin bitrade which is obtained from another one by deleting
its empty rows and empty columns, is called a {\sf
$k$--homogeneous Latin bitrade}, if in each row and each column it
contains exactly $k$ elements, and each element appears exactly
$k$ times. The number of filled cells in a Latin trade is referred to
as its {\sf volume}.
%
The following question is of interest.
\begin{question}
\label{existence}
For given $m$ and $k$,  $m \ge k$, does there exist a $k$--homogeneous Latin bitrade of volume $km$?
\end{question}
In the sequel we need some more notations and definitions.
Concepts not defined here may be found in~\cite{anderson90a}. We
can represent each Latin square as a set of $3$--tuples $L =
\{(i,j;k)\,|\,\mbox{element $k$ is located in position} \
(i,j)\}.$ A Latin bitrade $(T,T^{*})$ is said to be {\sf primary}
if whenever $(U,U^{*})$ is a Latin bitrade such that $U\subseteq
T$ and $U^{*}\subseteq T^{*}$, then $(T,T^{*})=(U,U^{*})$. A
Latin trade $T$ is said to be $\sf minimal$ if whenever
$(U,U^{*})$ is a Latin bitrade such that $U\subseteq T$, then
$T=U$. So if  $T$ is a minimal Latin trade in a  Latin bitrade
$(T,T^{*})$,  then $(T,T^{*})$  is a primary Latin bitrade. A
Latin bitrade of volume $4$ is called an {\sf intercalate}. In
Figure~\ref{2x2} an intercalate $(T,T^{*})$ is shown. The
elements of $T^{*}$ is written as subscripts in the same array as
$T$.
\def\arraystretch{1.1}
\begin{center}
\begin{tabular}
{|@{\hspace{1pt}}c@{\hspace{1pt}} |@{\hspace{1pt}}c@{\hspace{1pt}}
|@{\hspace{1pt}}c@{\hspace{1pt}} |@{\hspace{1pt}}c@{\hspace{1pt}}
|@{\hspace{1pt}}c@{\hspace{1pt}} |@{\hspace{1pt}}c@{\hspace{1pt}}
|} \hline
\m{1}{2}&\m{2}{1}\\\hline
\m{2}{1}&\m{1}{2}\\\hline
\end{tabular}
\end{center}
\begin{center}
\begin{figure}[ht]
\label{2x2} \vspace*{-7mm} \caption{An intercalate}
\end{figure}
\end{center}
\vspace*{-7mm}
We call a  Latin bitrade {\sf circulant} if it can be obtained from the elements of its first row, called
{\sf base row}, by permuting them diagonally.
See Figure~2.
%
\def\arraystretch{1.0}
\begin{center}
\begin{tabular}
{|@{\hspace{1pt}}c@{\hspace{1pt}} |@{\hspace{1pt}}c@{\hspace{1pt}}
|@{\hspace{1pt}}c@{\hspace{1pt}}|@{\hspace{1pt}}c@{\hspace{1pt}}
|@{\hspace{1pt}}c@{\hspace{1pt}}|} \hline
\m{2}{1}&\m{1}{3}&\m{3}{2}&\x.&\x.\\\hline
\x.&\m{3}{2}&\m{2}{4}&\m{4}{3}&\x.\\\hline
\x.&\x.&\m{4}{3}&\m{3}{5}&\m{5}{4}\\\hline
\m{1}{5}&\x.&\x.&\m{5}{4}&\m{4}{1}\\\hline
\m{5}{2}&\m{2}{1}&\x.&\x.&\m{1}{5}\\\hline
\end{tabular}
\hspace*{10mm}
\begin{tabular}
{|@{\hspace{1pt}}c@{\hspace{1pt}} |@{\hspace{1pt}}c@{\hspace{1pt}}
|@{\hspace{1pt}}c@{\hspace{1pt}}|@{\hspace{1pt}}c@{\hspace{1pt}}
|@{\hspace{5pt}}c@{\hspace{5pt}}|} \hline
\m{2}{1}&\m{1}{3}&\m{3}{2}&\x.&\x.\\\hline
\x.&$\searrow$&$\searrow$&$\searrow$&\x.\\\hline
\x.&\x.&\ &\ &\ \\\hline
\ &\x.&\x.&\ &\ \\\hline
\ &\ &\x.&\x.&\ \\\hline
\end{tabular}
\\\vspace*{-.01mm}
\end{center}
\begin{center}
\begin{figure}[ht]
\label{circulant} \vspace*{-6mm} \caption{A circulant
$3$--homogeneous Latin bitrade  of volume $15$ and  its base row}
\end{figure}
\end{center}
\begin{example}\label{k=4}
The following is a base row of a circulant $4$--homogeneous Latin
bitrade of volume $4m$ for $m > 4$:
$$D_m^{4}=\{(3,2)_1,(1,4)_2,(4,1)_3,(2,3)_4\}.$$
Note that since in a base row of a circulant Latin
bitrade $T= (T_1, T_2 )$, all the elements are in the {\em first} row, we
use the notation $(i,j)_c$ for  $(1,c;i) \in T_1$ and
$(1,c;j)\in T_2$.
\end{example}
It is proved in \cite{MR2170114} that
$3$--homogeneous Latin bitrades of volume $3m$ exist for all $m\geq
3$, and in~\cite{MR2139816} they have discussed  minimal $4$--homogeneous Latin bitrades. In~\cite{MR2220235},
among other results it is shown that
the answer for Question~\ref{existence} is positive for all $m \ge k$, where $3 \le k \le 8$. While there is
an error in Theorem~6 of~\cite{MR2220235}, but the results are valid and we will explain  this in the
last section (Section 4.1). The following results from~\cite{MR2220235} will be used in this paper.

\begin{oldtheorem} {\rm (\cite{MR2220235})}.\label{sum}
If $\ell \neq 2, 6$ and for each $k\in\{k_1,\ldots,k_{\ell}\}$
there exists a $k$--homogeneous Latin bitrade of volume $kp$,
then a $(k_1+\cdots+k_{\ell})$--homogeneous Latin bitrade of
volume $(k_1+\cdots+k_{\ell}){\ell}p$ exists.  {\rm(}Some $k_i$s
can possibly be zero{\rm)}.
\end{oldtheorem}
\begin{oldtheorem}  {\rm (\cite{MR2220235})}.\label{k(k+1)}
For each $k>2$, a $k$--homogeneous  Latin bitrade of volume $k(k+1)$
exists.
\end{oldtheorem}
For the case of $k=2$ the following holds.
\begin{oldtheorem} {\rm (\cite{MR2220235})}.
\label{2hom} For any $m\geq 1$, there exists a $2$--homogeneous
Latin bitrade of volume $2m$ if and only if $m$ is an even
integer.
\end{oldtheorem}
\begin{oldtheorem} {\rm (\cite{MR2220235})}.
\label{5m} For any $m=5{\ell}$ and $3\leq k \leq m$, there exists
a $k$--homogeneous Latin bitrade of volume $km$.
\end{oldtheorem}
\begin{oldtheorem}\label{2k-1} {\rm (\cite{MR2220235})}.
Consider an arbitrary integer $k$. If  for any ${k+1}\leq m\leq
2k-1$ there exists a $k$--homogeneous Latin bitrade of volume
$km$, then for any $m \geq k$ there exists a $k$--homogeneous
Latin bitrade of volume $km$.
\end{oldtheorem}
Here we prove that for each given odd integer $k\ge 3$ and for $m
\ge k$,  all $k$--homogeneous Latin bitrades of volume $km$ exist
and for all even integers $k>2$ and $m \ge \min\{k+u,
\frac{3k}{2}\}$, where $u$ is any odd  integer which divides $k$,
all $k$--homogeneous Latin bitrades of volume $km$ exist. We also
show that for $3\le k\leq 37$ and $m\geq k$, $k$--homogeneous
Latin bitrades of volume $km$ exist.

\section{Constructions and  general results}
We discuss our constructions depending on the parity of $k$.
\subsection{$k$ is odd}
\begin{theorem}\label{odd}
A $k$--homogeneous Latin bitrade of volume $km$ exists for all
odd integers $k$ and $m \ge k\ge 3$.
\end{theorem}
\begin{proof}{
Assume $k=2{\ell}+1$ and $m\geq k$. The following is a base row
of a circulant $k$--homogeneous Latin bitrade of volume $km$:
$$B_m^{2{\ell}+1}=
(\bigcup_{i=1}^{{\ell}+1}({\ell}+i,i)_{2i-1})\bigcup(\bigcup_{i=1}^{{\ell}}(i,{\ell}+1+i)_{2i}).$$

\vspace*{-1.1cm}}\end{proof}
\vspace*{0.1mm}
\begin{theorem}\label{ primary odd}
All constructed circulant $k$--homogeneous Latin bitrades in
Theorem~\ref{odd}, are primary.
\end{theorem}
\begin{proof}{Suppose $(T,T^{*})$ is the Latin bitrade constructed in the proof of
Theorem~\ref{odd}. Let $(U,U^{*})$ be a Latin bitrade such that
$U\subseteq T$ and $U^{*}\subseteq T^{*}$, we show that
$(U,U^{*})=(T,T^{*})$. Without loss of generality assume that
$(1,1;{\ell}+1) \in U$ and therefore $(1,1;1) \in U^{*}$. Since
$1$ must appear in the first row of $U$ and since $U \subseteq
T$, the only possibility is $(1,2;1) \in U$. Then we must have
$(1,2;{\ell}+2) \in U^{*}$. Similarly  $(1,3;{\ell}+2) \in U$,
thus $(1,3;2) \in U^{*}$. Following this process results that
$(1,2{\ell}+1;2{\ell}+1) \in U$,  and then
$(1,2{\ell}+1;{\ell}+1) \in U^{*}$. Therefore all the elements in
the first row of $T$ ($T^{*}$) are the same as all the elements
in the first row of $U$ ($U^{*}$). With the similar argument the
first column of $T$ ($T^{*}$) is the same as the first column of
$U$ ($U^{*}$). Finally this reasoning ends up showing that $U=T$
and $U^{*}=T^{*}$. }\end{proof}
%
%
%
%
\subsection{$k$ is even}
\begin{theorem}\label{even 3k/2}
A $k$--homogeneous Latin bitrade of volume $km$ exists for all
even integers $k>4$ and $m \ge \frac{3k}{2}.$
\end{theorem}
\begin{proof}{
Let $k=2a$ ($k>4$) and $m \ge \frac{3k}{2}$. The following is a
base row of a circulant $k$--homogeneous Latin bitrade of volume
$km$, when ${\ell}=a-2$:
$$B_m^{2{\ell}+1}\bigcup\{(3a-1,3a-2)_{2a-2},(3a-2,3a)_{2a-1},(3a,3a-1)_{2a}\}.$$

\vspace*{-1cm}}\end{proof}
\noindent {\bf Notation.} Note that a base row $B_m^{2{\ell}+1}$
was defined in Theorem~\ref{odd}.  We use a more general
notation, $B_m^{(r)(2{\ell}+1)}$, for a base row obtained from
$B_m^{2{\ell}+1}$ by adding $2(r-1)(2{\ell}+1)$ for both elements
in each cell of $B_m^{2{\ell}+1}$ and moving entry of each cell
$x$ to the cell $x+(r-1)(2{\ell}+1)$. Also for even $k>2$ we
denote by $C_m^{k(r)(2{\ell}+1)}$ a base row obtained from
$B_m^{2{\ell}+1}$, by adding $(2r-1)(2{\ell}+1)$ for both
elements in each cell of $B_m^{2{\ell}+1}$ and moving entry of
each cell $y$ to the cell $y+k/2+r(2{\ell}+1)$.
\begin{theorem}\label{even k+u}
A $k$--homogeneous Latin bitrade of volume $km$ exists for all
even integers $k>2$ and  $m \ge (k+u)$, where $u$ is any odd
integer greater than $1$ that divides $k$.
\end{theorem}
\begin{proof}{
If $u=2{\ell}+1$ then let $s=k/2u$. The following is a base row
of a circulant $k$--homogeneous Latin bitrade of volume $km$:
$$(\bigcup_{r=1}^{s+1} B_m^{(r)(2{\ell}+1)})\bigcup(\bigcup_{r=1}^{s-1}
C_m^{k(r)(2{\ell}+1)}).$$

\vspace*{-1cm}}\end{proof}

%
\section{More constructions}
The following theorem is very useful recursive construction.
\begin{theorem}\label{kmxln}
Let $m\geq k$ and $n\ge {\ell}$. If there exist a
$k$--homogeneous Latin bitrade of volume $km$, and  an
${\ell}$--homogeneous Latin bitrade of volume ${\ell}n$, then
there exists a $k{\ell}$--homogeneous Latin bitrade of volume
$(k{\ell})(mn).$
\end{theorem}
\begin{proof}{
 We construct a $k{\ell}$--homogeneous Latin bitrade of volume $(k{\ell})(mn)$
in the following way. Suppose $(T_1,T_2)$ is a  $k$--homogeneous
Latin bitrade of volume $km$. We replace each $i$ in $T_1$ and
$T_2$ with an ${\ell}$--homogeneous Latin trade of volume
${\ell}n$ whose elements are from 
$\{(i-1)n+1,(i-1)n+2,\ldots,in\}$; and the empty cells in $T_1$
and $T_2$  with an empty $n\times n$ array. As a result  we
obtain a $k{\ell}$--homogeneous Latin bitrade of volume
$(k{\ell})(mn).$
 }\end{proof}

\begin{example}\label{example3x5}
The existence of a
$2$--homogeneous Latin bitrade of volume $4$ (an intercalate), and  a
$3$--homogeneous Latin bitrade of volume $15$  imply the existence of a
$6$--homogeneous Latin bitrade of volume $60$. Indeed we take a Latin trade
of an intercalate of the following form:
\begin{center}
\begin{tabular}
{|@{\hspace{5pt}}c@{\hspace{5pt}} |@{\hspace{5pt}}c@{\hspace{5pt}}
|} \hline a&b\\\hline
 b&a\\\hline
\end{tabular}
\end{center}
then for  $i=1,2,3,4,5$ let $a=2(i-1)$ and $b=2(i-1)+1$, we replace them by the filled cells of the $3$--homogeneous
Latin bitrade of volume $15$  (of Figure~$2$) and obtain the following.
\def\arraystretch{1.3}
\begin{center}
\begin{tabular}
{||@{\hspace{2.5pt}}c@{\hspace{2.5pt}} |@{\hspace{2.5pt}}c@{\hspace{2.5pt}}
||@{\hspace{2.5pt}}c@{\hspace{2.5pt}} |@{\hspace{2.5pt}}c@{\hspace{2.5pt}}
||@{\hspace{2.5pt}}c@{\hspace{2.5pt}} |@{\hspace{2.5pt}}c@{\hspace{2.5pt}}
||@{\hspace{2.5pt}}c@{\hspace{2.5pt}} |@{\hspace{2.5pt}}c@{\hspace{2.5pt}}
||@{\hspace{2.5pt}}c@{\hspace{2.5pt}}
|@{\hspace{1pt}}c@{\hspace{1pt}} ||} \hline\hline
\m{2}{0}&\m{3}{1}&\m{0}{4}&\m{1}{5}&\m{4}{2}&\m{5}{3}&\x.&\x.&\x.&\x.\\\hline
\m{3}{1}&\m{2}{0}&\m{1}{5}&\m{0}{4}&\m{5}{3}&\m{4}{2}&\x.&\x.&\x.&\x.\\\hline\hline
\x.&\x.&\m{4}{2}&\m{5}{3}&\m{2}{6}&\m{3}{7}&\m{6}{4}&\m{7}{5}&\x.&\x.\\\hline
\x.&\x.&\m{5}{3}&\m{4}{2}&\m{3}{7}&\m{2}{6}&\m{7}{5}&\m{6}{4}&\x.&\x.\\\hline\hline
\x.&\x.&\x.&\x.&\m{6}{4}&\m{7}{5}&\m{4}{8}&\m{5}{9}&\m{8}{6}&\m{9}{7}\\\hline
\x.&\x.&\x.&\x.&\m{7}{5}&\m{6}{4}&\m{5}{9}&\m{4}{8}&\m{9}{7}&\m{8}{6}\\\hline\hline
\m{0}{8}&\m{1}{9}&\x.&\x.&\x.&\x.&\m{8}{6}&\m{9}{7}&\m{6}{0}&\m{7}{1}\\\hline
\m{1}{9}&\m{0}{8}&\x.&\x.&\x.&\x.&\m{9}{7}&\m{8}{6}&\m{7}{1}&\m{6}{0}\\\hline\hline
\m{8}{2}&\m{9}{3}&\m{2}{0}&\m{3}{1}&\x.&\x.&\x.&\x.&\m{0}{8}&\m{1}{9}\\\hline
\m{9}{3}&\m{8}{2}&\m{3}{1}&\m{2}{0}&\x.&\x.&\x.&\x.&\m{1}{9}&\m{0}{8}\\\hline\hline
\end{tabular}
\end{center}
\begin{center}
\begin{figure}[ht]
\label{3x5} \vspace*{-6mm} \caption{A $6$--homogeneous Latin bitrade of volume $60$}
\end{figure}
\end{center}
\end{example}
In the following we will improve the interval given in Theorem~\ref{2k-1}. First we need a lemma and a corollary.
\begin{lemma}\label{k+3}
A $k$--homogeneous Latin bitrade of volume $km$ exists for all
integers $k$ and $m = k+3.$
\end{lemma}
\begin{proof}{If $k$ is odd, the statement follows from Theorem
\ref{odd}. For $k=2\ell$, in each case in the following, we
introduce a base row of a circulant $k$--homogeneous Latin
bitrade of volume $km$, depending on the modulo classes of $k$.
First we define two types for the \underline{first row} in $T$:

{\bf Type I}. For $1 \le i \le \ell -1$, in the $(2i-1)$-th cell
($2i$-th cell, respectively)  we put $i$ ($\ell +2+i$,
respectively). In the $(k-1)$-th and $m$-th  cells we put $\ell$
and $\ell +2$, respectively.

{\bf Type II}. For $1 \le i \le \ell -1$, in the $(2i-1)$-th cell
($2i$-th cell, respectively)  we put $i$ ($\ell +2+i$,
respectively). In the $k$-th and $m$-th  cells we put $\ell +1$
and $\ell +2$, respectively.

Now we introduce the base rows.

\begin{enumerate}
\item  \qquad
$k \equiv 1 \ ({\rm mod \ 7})$

Let the first row of $T$ be as in Type I. For $T^*$, in the first
row and in the $(7i+3)$-th cell ($i\ge 0$ and $7i+3 < k$) we let
$a+4$ ({\rm mod  $m$}), where $a$ is the element of $T$ in the
same cell. Now in the $(k-1)$-th and $m$-th  cells of $T^*$ we put
$1$ and $\ell +4$, respectively. Finally in each  cell $c$ of the
first row in $T^*$ which is filled in $T$ but is so far empty in
$T^*$, we let the entry of  $(c+1)$-th cell of $T$.
\item  \qquad
$k \equiv 2 \ ({\rm mod \ 7})$

Let the first row of $T$ be as in Type I. For $T^*$, in the first
row and in the $(7i+4)$-th cell ($i\ge 0$ and $7i+4 < k$) we let
$a+4$ ({\rm mod  $m$}), where $a$ is the element of $T$ in the
same cell. Now in the $(k-1)$-th and $m$-th  cells of $T^*$ we put
$1$ and $3$, respectively. Finally in each  cell $c$ of the first
row in $T^*$ which is filled in $T$ but is so far empty in $T^*$,
we let the entry of  $(c+1)$-th cell of $T$.
\item  \qquad
$k \equiv 3 \ ({\rm mod \ 7})$

Let the first row of $T$ be as in Type II. For $T^*$, in the first
row and in the $(7i+3)$-th cell ($i\ge 0$ and $7i+3 < k$) we let
$a+4$ ({\rm mod  $m$}), where $a$ is the element of $T$ in the
same cell. Now in the $(k-2)$-th, $k$-th and $m$-th  cells of
$T^*$ we put $1$, $\ell +2$ and $\ell +4$, respectively. Finally
in each cell $c$ of the first row in $T^*$ which is filled in $T$
but is so far empty in $T^*$, we let the entry of  $(c+1)$-th cell
of $T$.
\item  \qquad
$k \equiv 4 \ ({\rm mod \ 7})$

Let the first row of $T$ be as in Type II. For $T^*$, in the first
row and in the $(7i+4)$-th cell ($i\ge 0$ and $7i+4 < k$) we let
$a+4$ ({\rm mod  $m$}), where $a$ is the element of $T$ in the
same cell. Now in the $(k-2)$-th, $k$-th and $m$-th  cells of
$T^*$ we put $1$, $\ell +2$ and $3$, respectively. Finally in
each cell $c$ of the first row in $T^*$ which is filled in $T$
but is so far empty in $T^*$, we let the entry of  $(c+1)$-th cell
of $T$.
\item  \qquad
$k \equiv 5 \ ({\rm mod \ 7})$

Let the first row of $T$ be as in Type I. For $T^*$, in the first
row and in the $(7i+r)$-th cell ($i\ge 0$, $r=1,2,3$ and $7i+r <
k$) we let $a+4$ ({\rm mod  $m$}), where $a$ is the element of
$T$ in the same cell. Now in the $(k-1)$-th and $m$-th  cells of
$T^*$ we put $\ell +2$ and $\ell +4$, respectively. Finally in
each cell $c$ of the first row in $T^*$ which is filled in $T$
but is so far empty in $T^*$, we let the entry of  $(c+1)$-th
cell of $T$.
\item  \qquad
$k \equiv 6 \ ({\rm mod \ 7})$

Let the first row of $T$ be as in Type I. For $T^*$, in the first
row and in the $(7i+r)$-th cell ($i\ge 0$, $r=2,4$ and $7i+r < k$)
we let $a+4$ ({\rm mod  $m$}), where $a$ is the element of $T$ in
the same cell. Now in the $(k-1)$-th and $m$-th  cells of $T^*$
we put $\ell +2$ and $3$, respectively. Finally in each cell $c$
of the first row in $T^*$ which is filled in $T$ but is so far
empty in $T^*$, we let the entry of  $(c+1)$-th cell of $T$.
\item  \qquad
$k \equiv 0 \ ({\rm mod \ 7})$

Let the first row of $T$ be as in Type I. For $T^*$, in the first
row and in the $(7i+3)$-th cell ($i\ge 0$ and $7i+3 < k$) we let
$a+4$ ({\rm mod  $m$}), where $a$ is the element of $T$ in the
same cell. Now in the $(k-1)$-th and $m$-th  cells of $T^*$ we put
$\ell +2$ and $\ell +4$, respectively. Finally in each  cell $c$
of the first row in $T^*$ which is filled in $T$ but is so far
empty in $T^*$, we let the entry of  $(c+1)$-th cell of $T$.
\end{enumerate}
\vspace*{-1cm}}\end{proof}
\begin{corollary}\label{k+6}
A $k$--homogeneous Latin bitrade of volume $km$ exists for all
integers $k$ and $m = k+6$.
\end{corollary}
\begin{proof}{
If $k$ is an odd integer, then the statement follows from
Theorem~\ref{odd}. In case $k=2{\ell}$ we know that by
Lemma~\ref{k+3} there exist an ${\ell}$--homogeneous Latin
bitrade of volume ${\ell}({\ell}+3)$ and a $2$--homogeneous Latin
bitrade of volume $4$, therefore by Theorem~\ref{kmxln} there
exists a $2{\ell}$--homogeneous Latin bitrade of volume
$2{\ell}(2{\ell}+6)$.

\vspace*{-0.7cm}}\end{proof}
\begin{lemma}\label{k+2,4}
A $k$--homogeneous Latin bitrade of volume $km$ exists for all
integers $k$ and $m = k+2,k+4.$
\end{lemma}
\begin{proof}{
If $k$ is an odd integer then the statement follows from
Theorem~\ref{odd}.
Let $k=2{\ell}$.
\begin{itemize}
\item
$m=k+2$

By Theorem~\ref{k(k+1)}  and Theorem~\ref{kmxln} there exists a
$2{\ell}$--homogeneous Latin bitrade of volume
$2{\ell}(2{\ell}+2)$.

\item $m=k+4$

By previous case and by Theorem~\ref{kmxln} there exists a
$2{\ell}$--homogeneous Latin bitrade of volume
$2{\ell}(2{\ell}+4)$.

\end{itemize}
 \vspace*{-1cm}}\end{proof}
The following theorem follows from Theorem~\ref{even 3k/2},  Lemmas~\ref{k+3}~and~\ref{k+2,4}.
\begin{theorem}\label{k+5}
Let  $k$ be an integer. If  for all $m$, ${k+5}\leq m<
3k/2$, there exists a $k$--homogeneous Latin bitrade of volume
$km$, then for any $m \geq k$ there exists a $k$--homogeneous
Latin bitrade of volume $km$.
\end{theorem}
%
\section{The intervals}
From Theorems~\ref{even k+u} and ~\ref{k+5} a result follows which
is very useful in the constructions of the needed bitrades:
\begin{corollary}\label{3k',5k'}
If $k$ is a multiple of $3$ or $5$, then there exists a  $k$--homogeneous
Latin bitrade of volume $km$ for all $m \geq k$.
\end{corollary}
\subsection{$2\leq k\leq8$}
The `proof' of Theorem~6 in \cite{MR2220235} is false, but we may
apply   Theorem~\ref{kmxln} above, and Theorem~\ref{sum}
  (Theorem~1 in \cite{MR2220235}) to correct all results in that paper where ever its Theorem~6 is used.
  For example for the Case~1 in the proof of Theorem~9
(in~\cite{MR2220235}), we take the following parameters in
Theorem~\ref{sum} (Theorem~1 in \cite{MR2220235}): $k_i=5$ for $1
\leq i \leq {\ell}^{'} $, \ $k_i=0$ for ${\ell}^{'}+1 \leq i \leq
{\ell}$ and $p=5$. Or for the Case~4 in the proof of Main
Theorem~2 (in \cite{MR2220235}), since there exist a
$4$--homogeneous Latin bitrade of volume $24$ and a
$2$--homogeneous Latin bitrade of volume $4$, so by
Theorem~\ref{kmxln} above, there exists an $8$--homogeneous Latin
bitrade of volume $96$.

So for the interval $2\leq k\leq8$, Example~\ref{k=4}, Theorem~\ref{2hom}  and the following theorem answer Question~\ref{existence}.
\begin{oldtheorem}
{\rm\bf (Main Theorem~2 of ~\cite{MR2220235}).}
\label{kolli} For any $k$, $5\leq k\leq 8$ and  $m\geq k $, there
exists a $k$--homogeneous Latin bitrade of volume $km$.
\end{oldtheorem}
\subsection{$9\leq k\leq37$}
\begin{theorem}\label{37}
If $9\leq k\leq 37 $ then there exists a $k$--homogeneous Latin
bitrade of volume $km$ for any $m \geq k$.
\end{theorem}
\begin{proof}{
Note that the case $k$ odd  follows by Theorem~\ref{odd}. The
cases$\linebreak$ $k=10,12,18,20,24,30,36$ follow by
Corollary~\ref{3k',5k'}. For $k=14$, by$\linebreak$
Theorem~\ref{k+5} we only need to show for $m=19$ and $m=20$.

For $m=20$ we apply Theorem~\ref{5m}. The following base row is for $m=19$:

$D_{19}^{14}=\{(1,11)_1,(11,2)_2,(2,12)_3,(12,3)_4,(3,13)_5,(13,4)_6,
(4,14)_7,$ \\
\hspace*{17.7mm}$(14,5)_8,(5,1)_9,(6,7)_{11},
(7,8)_{13},(8,9)_{15},(9,10)_{17},(10,6)_{19}\}.$

For $k=16$, again by Theorem~\ref{k+5}, it suffices to show the
existence of $16$--homogeneous Latin bitrades of volume $16m$,
where $21 \le m \le 23$. The case $m=21$ follows from
Theorem~\ref{sum} by letting $k_1=4$,  $k_2=k_3=6$ and $p=7$. The
case $m=22$   follows from Theorem~\ref{kmxln}. 
And the following base row is for $m=23$:

$D_{23}^{16}=\{(1,13)_1,(13,2)_2,(2,14)_3,(14,3)_4,(3,15)_5,(15,4)_6,
(4,16)_7,$ \\
\hspace*{17.7mm}$(16,5)_8,(5,17)_9,(17,6)_{10},(6,1)_{11},(7,8)_{13},(8,10)_{16},
(10,11)_{19},$ \\
\hspace*{17.7mm}$(11,12)_{21},(12,7)_{23}\}$.

Similarly for $k=22,26,28,32,34$ we include the base rows in
the $\linebreak$  Appendix for odd integers ${k+5}\leq m< 3k/2$ such
that $m\ne 5{\ell}$. By Theorems~\ref{kmxln} and~\ref{5m} the
proof is complete. }\end{proof}
The results above motivates us to conjecture that:
\begin{conjecture}
\label{existenceconjecture} For all $m$ and $k$,  $m \ge k\ge
3$,  there exists a $k$--homogeneous Latin bitrade of volume $km$.
\end{conjecture}
%
\section*{Appendix}
The followings are base rows of bitrades needed in the proof of Theorem~\ref{37}:
\begin{itemize}
\item{$\bf k=22$\\
 $D_{27}^{22}=\{(1,19)_{1},(3,2)_{2},(2,4)_{3},(6,3)_{4},(8,7)_{5},(4,9)_{6},(11,5)_{7},(5,12)_{8},\\
\hspace*{12.6mm}(14,17)_{9},(16,27)_{10},(7,18)_{11},(19,6)_{12},(21,10)_{13},(9,24)_{14},\\
\hspace*{12.6mm}(24,21)_{15},(10,11)_{16},(27,13)_{17},(12,15)_{22},(15,1)_{23},(13,16)_{24},\\
\hspace*{12.6mm}(18,14)_{25},(17,8)_{26}\}$

$D_{29}^{22}=\{(1,21)_{1},(3,2)_{2},(2,4)_{3},(6,3)_{4},(8,7)_{5},(4,9)_{6},(11,5)_{7},(5,12)_{8},$\\
\hspace*{12.6mm}$(14,16)_{9},(16,15)_{10},(7,17)_{11},(19,8)_{12},(21,6)_{13},(9,24)_{14},$\\
\hspace*{12.6mm}$(24,10)_{15},(10,13)_{16},(27,11)_{17},(29,27)_{18},(12,1)_{19},(13,29)_{21},$\\
\hspace*{12.6mm}$(15,14)_{24},(17,19)_{27}\}$

$D_{31}^{22}=\{(1,24)_{1},(3,2)_{2},(2,4)_{3},(6,3)_{4},(8,7)_{5},(4,9)_{6},(11,5)_{7},(5,12)_{8},$\\
\hspace*{12.6mm}$(14,16)_{9},(16,15)_{10},(7,1)_{11},(19,21)_{12},(21,19)_{13},(9,11)_{14},$\\
\hspace*{12.6mm}$(24,10)_{15},(10,27)_{16},(27,8)_{17},(29,14)_{18},(12,13)_{19},(13,29)_{21},$\\
\hspace*{12.6mm}$(15,17)_{24},(17,6)_{27}\}$ }
\item{$\bf k=26$\\
$D_{31}^{26}=\{(1,19)_{1},(3,2)_{2},(2,4)_{3},(6,3)_{4},(8,7)_{5},(4,9)_{6},(11,5)_{7},(5,12)_{8},$\\
\hspace*{12.6mm}$(14,6)_{9},(16,15)_{10},(7,17)_{11},(19,1)_{12},(21,20)_{13},(9,22)_{14},$\\
\hspace*{12.6mm}$(24,10)_{15},(10,27)_{16},(27,13)_{17},(29,11)_{18},(31,29)_{19},(12,31)_{22}$\\
\hspace*{12.6mm}$,(13,16)_{25},(15,14)_{26},(18,8)_{27},(20,18)_{28},(22,21)_{29},(17,24)_{30}\}$

$D_{33}^{26}=\{(1,22)_{1},(3,2)_{2},(2,4)_{3},(6,3)_{4},(8,7)_{5},(4,9)_{6},(11,5)_{7},(5,12)_{8},$\\
\hspace*{12.6mm}$(14,6)_{9},(16,1)_{10},(7,17)_{11},(19,20)_{12},(21,18)_{13},(9,24)_{14},$\\
\hspace*{12.6mm}$(24,10)_{15},(10,27)_{16},(27,29)_{17},(29,14)_{18},(12,11)_{19},(32,13)_{20}$\\
\hspace*{12.6mm}$,(13,15)_{21},(15,32)_{25},(18,16)_{29},(17,19)_{30},(22,21)_{31},(20,8)_{32}\}$

$D_{37}^{26}=\{(1,27)_{1},(3,2)_{2},(2,4)_{3},(6,3)_{4},(8,7)_{5},(4,9)_{6},(11,5)_{7},(5,12)_{8},$\\
\hspace*{12.6mm}$(14,6)_{9},(16,15)_{10},(7,19)_{11},(19,18)_{12},(21,24)_{13},(9,21)_{14},$\\
\hspace*{12.6mm}$(24,10)_{15},(10,29)_{16},(27,8)_{17},(29,32)_{18},(12,13)_{19},(32,16)_{20},$\\
\hspace*{12.6mm}$(13,11)_{21},(35,14)_{22},(37,35)_{23},(15,17)_{24},(17,37)_{27},(18,1)_{29}\}$
}
\item{$\bf k=28$\\
$D_{33}^{28}=\{(1,20)_{1},(3,2)_{2},(2,4)_{3},(6,3)_{4},(8,7)_{5},(4,9)_{6},(11,5)_{7},(5,12)_{8},$\\
\hspace*{12.6mm}$(14,6)_{9},(16,15)_{10},(7,17)_{11},(19,8)_{12},(21,1)_{13},(9,24)_{14},$\\
\hspace*{12.6mm}$(24,23)_{15},(10,11)_{16},(27,29)_{17},(29,27)_{18},(31,13)_{19},(33,31)_{20},$\\
\hspace*{12.6mm}$(12,14)_{21},(13,33)_{26},(15,19)_{27},(17,18)_{28},(22,16)_{29},(20,10)_{30},$\\
\hspace*{12.6mm}$(23,22)_{31},(18,21)_{32}\}$

$D_{37}^{28}=\{(1,27)_{1},(3,2)_{2},(2,4)_{3},(6,3)_{4},(8,7)_{5},(4,9)_{6},(11,5)_{7},(5,12)_{8},$\\
\hspace*{12.6mm}$(14,6)_{9},(16,15)_{10},(7,17)_{11},(19,37)_{12},(21,22)_{13},(9,21)_{14},$\\
\hspace*{12.6mm}$(24,10)_{15},(10,24)_{16},(27,29)_{17},(29,11)_{18},(12,32)_{19},(32,14)_{20}$\\
\hspace*{12.6mm}$,(13,35)_{21},(35,18)_{22},(37,13)_{23},(15,16)_{24},(17,1)_{27},(18,20)_{29}$\\
\hspace*{12.6mm}$,(20,19)_{32},(22,8)_{35}\}$

$D_{39}^{28}=\{(1,27)_{1},(3,2)_{2},(2,4)_{3},(6,3)_{4},(8,7)_{5},(4,9)_{6},(11,5)_{7},(5,12)_{8},$\\
\hspace*{12.6mm}$(14,6)_{9},(16,15)_{10},(7,17)_{11},(19,20)_{12},(21,1)_{13},(9,24)_{14},$\\
\hspace*{12.6mm}$(24,22)_{15},(10,8)_{16},(27,13)_{17},(29,11)_{18},(12,14)_{19},(32,29)_{20},$\\
\hspace*{12.6mm}$(13,32)_{21},(35,16)_{22},(37,35)_{23},(15,37)_{24},(17,18)_{27},(18,19)_{29}$\\
\hspace*{12.6mm}$,(20,21)_{32},(22,10)_{35}\}$

$D_{41}^{28}=\{(1,32)_{1},(3,2)_{2},(2,4)_{3},(6,3)_{4},(8,7)_{5},(4,9)_{6},(11,5)_{7},(5,12)_{8},$\\
\hspace*{12.6mm}$(14,6)_{9},(16,15)_{10},(7,17)_{11},(19,21)_{12},(21,20)_{13},(9,29)_{14},$\\
\hspace*{12.6mm}$(24,27)_{15},(10,24)_{16},(27,11)_{17},(29,13)_{18},(12,8)_{19},(32,16)_{20},$\\
\hspace*{12.6mm}$(13,35)_{21},(35,10)_{22},(37,14)_{23},(15,37)_{24},(40,18)_{25},(17,19)_{27}$\\
\hspace*{12.6mm}$,(18,40)_{29},(20,1)_{32}\}$ }
\item{$\bf k=32$\\

$D_{37}^{32}=\{(1,22)_{1},(3,2)_{2},(2,4)_{3},(6,3)_{4},(8,7)_{5},(4,9)_{6},(11,5)_{7},(5,12)_{8},$\\
\hspace*{12.6mm}$(14,6)_{9},(16,15)_{10},(7,17)_{11},(19,8)_{12},(21,20)_{13},(9,1)_{14},$\\
\hspace*{12.6mm}$(24,10)_{15},(10,25)_{16},(27,29)_{17},(29,26)_{18},(12,13)_{19},(32,35)_{20},$\\
\hspace*{12.6mm}$(35,32)_{21},(37,36)_{22},(36,16)_{23},(13,37)_{27},(17,14)_{29},(15,18)_{30},$\\
\hspace*{12.6mm}$(18,23)_{31},(22,21)_{32},(25,19)_{33},(23,24)_{34},(26,11)_{35},(20,27)_{36}\}$\\

$D_{39}^{32}=\{(1,24)_{1},(3,2)_{2},(2,4)_{3},(6,3)_{4},(8,7)_{5},(4,9)_{6},(11,5)_{7},(5,12)_{8},$\\
\hspace*{12.6mm}$(14,6)_{9},(16,15)_{10},(7,17)_{11},(19,8)_{12},(21,20)_{13},(9,22)_{14},$\\
\hspace*{12.6mm}$(24,39)_{15},(10,25)_{16},(27,11)_{17},(29,32)_{18},(12,29)_{19},(32,13)_{20},$\\
\hspace*{12.6mm}$(13,16)_{21},(35,14)_{22},(37,35)_{23},(39,37)_{24},(15,1)_{25},(17,18)_{28},$\\
\hspace*{12.6mm}$(18,19)_{33},(20,21)_{34},(26,23)_{35},(23,27)_{36},(25,26)_{37},(22,10)_{38}\}$\\

$D_{41}^{32}=\{(1,26)_{1},(3,2)_{2},(2,4)_{3},(6,3)_{4},(8,7)_{5},(4,9)_{6},(11,5)_{7},(5,12)_{8},$\\
\hspace*{12.6mm}$(14,6)_{9},(16,15)_{10},(7,17)_{11},(19,8)_{12},(21,20)_{13},(9,22)_{14},$\\
\hspace*{12.6mm}$(24,1)_{15},(10,25)_{16},(27,11)_{17},(29,32)_{18},(12,29)_{19},(32,10)_{20},$\\
\hspace*{12.6mm}$(13,16)_{21},(35,13)_{22},(37,35)_{23},(15,37)_{24},(40,18)_{25},(17,19)_{27},$\\
\hspace*{12.6mm}$(18,40)_{29},(20,21)_{32},(22,24)_{37},(25,23)_{38},(23,27)_{39},(26,14)_{40}\}$\\

$D_{43}^{32}=\{(1,29)_{1},(3,2)_{2},(2,4)_{3},(6,3)_{4},(8,7)_{5},(4,9)_{6},(11,5)_{7},(5,12)_{8},$\\
\hspace*{12.6mm}$(14,6)_{9},(16,15)_{10},(7,17)_{11},(19,8)_{12},(21,1)_{13},(9,23)_{14},$\\
\hspace*{12.6mm}$(24,22)_{15},(10,27)_{16},(27,25)_{17},(29,13)_{18},(12,32)_{19},(32,14)_{20},$\\
\hspace*{12.6mm}$(13,10)_{21},(35,37)_{22},(37,35)_{23},(15,40)_{24},(40,16)_{25},(42,18)_{26},$\\
\hspace*{12.6mm}$(17,20)_{27},(18,19)_{29},(20,42)_{32},(22,21)_{35},(23,24)_{37},(25,11)_{40}\}$\\

$D_{47}^{32}=\{(1,35)_{1},(3,2)_{2},(2,4)_{3},(6,3)_{4},(8,7)_{5},(4,9)_{6},(11,5)_{7},(5,12)_{8},$\\
\hspace*{12.6mm}$(14,6)_{9},(16,15)_{10},(7,17)_{11},(19,8)_{12},(21,20)_{13},(9,24)_{14},$\\
\hspace*{12.6mm}$(24,23)_{15},(10,11)_{16},(27,29)_{17},(29,27)_{18},(12,32)_{19},(32,13)_{20},$\\
\hspace*{12.6mm}$(13,10)_{21},(35,37)_{22},(37,40)_{23},(15,16)_{24},(40,19)_{25},(42,14)_{26},$\\
\hspace*{12.6mm}$(17,18)_{27},(45,42)_{28},(18,45)_{29},(20,22)_{32},(22,21)_{35},(23,1)_{37}\}$
}
\item{$\bf k=34$\\

$D_{39}^{34}=\{(1,24)_{1},(3,2)_{2},(2,4)_{3},(6,3)_{4},(8,7)_{5},(4,9)_{6},(11,5)_{7},(5,12)_{8},$\\
\hspace*{12.6mm}$(14,6)_{9},(16,15)_{10},(7,17)_{11},(19,8)_{12},(21,20)_{13},(9,22)_{14},$\\
\hspace*{12.6mm}$(24,10)_{15},(10,25)_{16},(27,39)_{17},(29,28)_{18},(12,13)_{19},(32,11)_{20},$\\
\hspace*{12.6mm}$(34,32)_{21},(37,34)_{22},(39,38)_{23},(38,37)_{24},(13,1)_{26},(15,16)_{30},$\\
\hspace*{12.6mm}$(17,21)_{31},(20,19)_{32},(23,18)_{33},(25,23)_{34},(18,27)_{35},(28,29)_{36},$\\
\hspace*{12.6mm}$(26,14)_{37},(22,26)_{38}\}$  \\

$D_{41}^{34}=\{(1,25)_{1},(3,2)_{2},(2,4)_{3},(6,3)_{4},(8,7)_{5},(4,9)_{6},(11,5)_{7},(5,12)_{8},$\\
\hspace*{12.6mm}$(14,6)_{9},(16,15)_{10},(7,17)_{11},(19,8)_{12},(21,20)_{13},(9,22)_{14},$\\
\hspace*{12.6mm}$(24,10)_{15},(10,41)_{16},(27,26)_{17},(29,28)_{18},(12,13)_{19},(32,35)_{20},$\\
\hspace*{12.6mm}$(13,32)_{21},(35,14)_{22},(37,16)_{23},(39,37)_{24},(41,39)_{25},(15,1)_{26},$\\
\hspace*{12.6mm}$(17,18)_{27},(18,19)_{33},(22,23)_{35},(20,21)_{36},(28,24)_{37},(26,27)_{38},$\\
\hspace*{12.6mm}$(25,29)_{39},(23,11)_{40}\}$   \\

$D_{43}^{34}=\{(1,27)_{1},(3,2)_{2},(2,4)_{3},(6,3)_{4},(8,7)_{5},(4,9)_{6},(11,5)_{7},(5,12)_{8},$\\
\hspace*{12.6mm}$(14,6)_{9},(16,15)_{10},(7,17)_{11},(19,8)_{12},(21,20)_{13},(9,22)_{14},$\\
\hspace*{12.6mm}$(24,42)_{15},(10,1)_{16},(27,26)_{17},(29,28)_{18},(12,14)_{19},(32,35)_{20},$\\
\hspace*{12.6mm}$(13,32)_{21},(35,10)_{22},(37,16)_{23},(15,40)_{24},(40,37)_{25},(42,18)_{26},$\\
\hspace*{12.6mm}$(17,21)_{27},(18,19)_{29},(20,23)_{32},(22,24)_{37},(23,25)_{39},(26,29)_{40},$\\
\hspace*{12.6mm}$(28,11)_{41},(25,13)_{42}\}$    \\

$D_{47}^{34}=\{(1,32)_{1},(3,2)_{2},(2,4)_{3},(6,3)_{4},(8,7)_{5},(4,9)_{6},(11,5)_{7},(5,12)_{8},$\\
\hspace*{12.6mm}$(14,6)_{9},(16,15)_{10},(7,17)_{11},(19,8)_{12},(21,20)_{13},(9,22)_{14},$\\
\hspace*{12.6mm}$(24,25)_{15},(10,1)_{16},(27,29)_{17},(29,27)_{18},(12,13)_{19},(32,37)_{20},$\\
\hspace*{12.6mm}$(13,10)_{21},(35,14)_{22},(37,18)_{23},(15,35)_{24},(40,16)_{25},(42,40)_{26},$\\
\hspace*{12.6mm}$(17,42)_{27},(45,21)_{28},(18,19)_{29},(20,45)_{32},(22,23)_{35},(23,24)_{37},$\\
\hspace*{12.6mm}$(25,26)_{40},(26,11)_{42}\}$    \\

$D_{49}^{34}=\{(1,35)_{1},(3,2)_{2},(2,4)_{3},(6,3)_{4},(8,7)_{5},(4,9)_{6},(11,5)_{7},(5,12)_{8},$\\
\hspace*{12.6mm}$(14,6)_{9},(16,15)_{10},(7,17)_{11},(19,8)_{12},(21,20)_{13},(9,27)_{14},$\\
\hspace*{12.6mm}$(24,23)_{15},(10,25)_{16},(27,29)_{17},(29,13)_{18},(12,37)_{19},(32,10)_{20},$\\
\hspace*{12.6mm}$(13,32)_{21},(35,14)_{22},(37,11)_{23},(15,18)_{24},(40,42)_{25},(42,40)_{26},$\\
\hspace*{12.6mm}$(17,16)_{27},(45,19)_{28},(18,22)_{29},(48,45)_{30},(20,48)_{32},(22,21)_{35},$\\
\hspace*{12.6mm}$(23,24)_{37},(25,1)_{40}\}$ }
\end{itemize}
\newpage

\noindent {\bf Acknowledgement.} This research was in part
supported by a grant (\#86050213) from the
Department of Mathematics of
 Institute for Studies in Theoretical Physics and Mathematics (IPM)
  P. O. Box 19395-5746, Tehran, I. R. Iran. We also appreciate
 the help of Amir Hooshang Hosseinpoor for his computer
 programming.
%


\begin{thebibliography}{1}

\bibitem{anderson90a}
Ian Anderson.
\newblock {\em Combinatorial designs: Construction Methods}.
\newblock John Wiley \& Sons, Inc., New York, 1990.

\bibitem{MR2220235}
Richard Bean, Hoda Bidkhori, Maryam Khosravi, and E.~S. Mahmoodian.
\newblock {$k$}-homogeneous {L}atin trades.
\newblock {\em Bayreuth. Math. Schr.}, 74:7--18, 2005.

\bibitem{MR2041871}
Elizabeth~J. Billington.
\newblock Combinatorial trades: a survey of recent results.
\newblock In {\em Designs, 2002}, volume 563 of {\em Math. Appl.}, pages
  47--67. Kluwer Acad. Publ., Boston, MA, 2003.

\bibitem{MR2170114}
Nicholas Cavenagh, Diane Donovan, and Ale{\v{s}} Dr{\'a}pal.
\newblock 3-homogeneous {L}atin trades.
\newblock {\em Discrete Math.}, 300(1-3):57--70, 2005.

\bibitem{MR2139816}
Nicholas Cavenagh, Diane Donovan, and Ale{\v{s}} Dr{\'a}pal.
\newblock 4-homogeneous {L}atin trades.
\newblock {\em Australas. J. Combin.}, 32:285--303, 2005.

\bibitem{CavenaghMathSlovac}
Nicholas~J. Cavenagh.
\newblock The theory and application of {L}atin bitrades: a survey.
\newblock {\em Math. Slovac.}, to appear.

\bibitem{MR2048415}
A.~D. Keedwell.
\newblock Critical sets in {L}atin squares and related matters: an update.
\newblock {\em Util. Math.}, 65:97--131, 2004.

\end{thebibliography}
%
\def\cprime{$'$}

\end{document}